\documentclass[11pt]{article}

\usepackage[T1]{fontenc}
\usepackage[utf8]{inputenc}
\usepackage[english]{babel}

\usepackage{amsmath,amssymb,amsthm,mathtools}
\usepackage{enumitem}
\usepackage[margin=1in]{geometry}
\usepackage{hyperref}
\usepackage{csquotes}
\newcommand{\virgolette}[1]{\enquote{#1}}

\theoremstyle{definition}
\newtheorem{definition}{Definition}[section]
\newtheorem{remark}[definition]{Remark}

\theoremstyle{plain}
\newtheorem{theorem}[definition]{Theorem}

\newtheorem{lemma}[definition]{Lemma}
\newtheorem{corollary}[definition]{Corollary}
\newtheorem{claim}[definition]{Claim}

\numberwithin{equation}{section}

\title{A $wtt$-introimmune set in \texorpdfstring{$\Pi^0_1$}{Pi01} and introimmunity for several reducibilities}

\author{
Patrizio Cintioli\\
Mathematics Division, School of Sciences and Technology\\
University of Camerino, Italy\\
\texttt{patrizio.cintioli@unicam.it}
}

\date{}

\begin{document}

\maketitle

\date{}

\begin{abstract}
We prove that there exists a weak truth-table introimmune set in the class $\Pi^0_1$, settling the question left open in previous work of whether the known $\Delta^0_2$ existence result can be improved to $\Pi^0_1$.
Since $\Sigma^0_1$ sets cannot be immune, this is best possible for weak truth-table introimmunity.
We also study introimmunity for Jockusch's bounded-search reducibility $\le_{bs}$ and Andersen's Dartmouth reducibility $\le_D$, proving the existence of $\Delta^0_2$ sets that are $bs$-introimmune and $D$-introimmune; hence there also exists a $\Delta^0_2$ $D^+$-introimmune set.
We next consider the classical reducibility $\le_Q$, which is not contained in $\le_T$ on all subsets of $\omega$.
We show that no infinite $\Pi^0_1$ set is $Q$-introimmune, while a $\Delta^0_2$ $Q$-introimmune set does exist.
Thus the existence of $\Delta^0_2$ $Q$-introimmune sets is best possible within the arithmetical hierarchy.
Finally, for enumeration reducibility $\le_e$, we show that no infinite $\Pi^1_1$ set is $e$-introimmune, although $e$-introimmune sets do exist in the unrestricted sense.
The proofs combine finite-injury priority arguments with dynamic spacing methods for $\le_{wtt}$, $\le_{bs}$, and $\le_D$, a bit-by-bit finite-extension construction for $\le_Q$, and an application of Soare's abstract existence theorem in the enumeration case.
\end{abstract}



\maketitle

\section{Introduction}

The notion of introimmunity may be viewed as a strong opposite to introreducibility.
Indeed, while a set fails to be introreducible as soon as one infinite subset does not compute it, introimmunity requires that no coinfinite subset computes the ambient set via the reducibility under consideration.
Recent work on introreducibility and related notions may be found in \cite{ghtpt21}.

The existence of sets without subsets of higher Turing degree was first established by Soare \cite{so69}.
Concerning their definability, Jockusch proved that such sets cannot be arithmetical \cite{jo73}, and Simpson later strengthened this by showing that they cannot be hyperarithmetical \cite{si78}.
Following \cite{cs03}, where this terminology was introduced in the polynomial-time setting, given a reducibility $\le_r$, an infinite set $A \subseteq \omega$ is called $r$-\emph{introimmune} if for every subset $B \subseteq A$ with $\lvert A\setminus B\rvert=\infty$ we have $A\not\le_r B$.
Whenever $\le_r$ is transitive, $r$-introimmunity ensures that $A$ has no subset of strictly higher $r$-degree.
Moreover, if $\le_r\Rightarrow\le_s$, then every $s$-introimmune set is automatically $r$-introimmune.

A classical line of research asks how low in the arithmetical hierarchy such sets can occur for reducibilities below $\le_T$.
For the many-one reducibility $\le_m$, it is known that $m$-introimmune sets exist in the class $\Pi^0_1$ \cite{ci05,ci11}.
For truth-table type reducibilities, the existence of $c$-introimmune sets in $\Delta^0_4$ was proved in \cite{cs03}, where $\le_c$ is the conjunctive reducibility, and Ambos-Spies\footnote{K.~Ambos-Spies informed the author in private correspondence that the published argument for Theorem~3 of \cite{am03} requires correction, and very kindly communicated a corrected version.} \cite{am03}
later improved this to the existence of $tt$-introimmune sets in $\Delta^0_2$.
Further structural properties of degrees of sets having no subsets of higher $m$- and $tt$-degree were investigated in \cite{ci21}.

In \cite{ci18}, we turned to the weak truth-table reducibility $\le_{wtt}$ (also called bounded Turing reducibility), proving the existence of $wtt$-introimmune sets in $\Delta^0_2$.
That paper left open whether the optimal arithmetical bound could be pushed down to $\Pi^0_1$.

A further motivation for the present work comes from the landscape of reducibilities below Turing reducibility.
Earlier investigations naturally concentrated on transitive reducibilities, since only in that setting does one obtain a corresponding degree structure and can interpret introimmunity in degree-theoretic terms.
However, the picture below $\le_T$ is subtler than the standard chain
\[
\le_1 \Rightarrow \le_m \Rightarrow \le_{tt} \Rightarrow \le_{wtt} \Rightarrow \le_T
\]
might suggest.
The literature contains natural notions which are intermediate only on restricted classes, or which fail to be transitive and therefore do not induce degrees on all subsets of $\omega$.
This makes it necessary to distinguish carefully between the degree-theoretic reading of introimmunity and the more basic reducibility-theoretic condition $A\not\le_r B$ for coinfinite subsets $B\subseteq A$.
In particular, Andersen introduced the Dartmouth reducibilities $\le_D$ and $\le_{D^+}$, showing that $\le_{D^+}$ is a transitive reducibility strictly below $\le_T$ and incomparable with $\le_{wtt}$ \cite{an08}; see also the corrigendum to \cite{ci18} in \cite{ci20corr}.
This suggests enlarging the investigation from the usual $wtt$-chain to reducibilities which are natural in this context even when they do not carry a degree structure.

Our first main result is the existence of a $wtt$-introimmune set in $\Pi^0_1$.
This gives the lowest possible arithmetical level for weak truth-table introimmunity.
The main obstacle is that the $\Delta^0_2$ construction of \cite{ci18} relied on a $K$-computable dominating function to impose in advance sufficiently large gaps between successive elements of the set, thereby isolating the bounded uses of the relevant $wtt$-computations.
In the present $\Pi^0_1$ setting the approximation is strictly subtractive, so such gaps cannot be precomputed.
To overcome this, we develop a \emph{dynamic spacing} technique: whenever the computable use of a convergent $wtt$-reduction threatens the next candidate marker, we permanently remove all intermediate elements.
Because removals in a $\Pi^0_1$ construction are irreversible, this clearing process creates the required separation once and for all and makes the diagonalization possible within a finite-injury priority argument.

Further main results concern reducibilities lying outside the usual $wtt$-chain.
In Section \ref{sec:bs_D} we consider Jockusch's bounded-search reducibility $\le_{bs}$ \cite{jo72} and Andersen's Dartmouth reducibility $\le_D$ \cite{an08}.
Unlike $wtt$-reductions, neither $bs$-reductions nor $D$-reductions come equipped with a computable a priori use bound, so a purely subtractive $\Pi^0_1$ construction appears substantially more rigid for them.
Nevertheless, by letting the halting problem dynamically compute the relevant spacing bounds, we prove the existence of $\Delta^0_2$ sets that are $bs$-introimmune and $D$-introimmune.
Since $\le_{D^+}\Rightarrow\le_D$, the latter result automatically yields a $D^+$-introimmune set as well.
This leaves open whether the $\Delta^0_2$ upper bound for $\le_{bs}$ and $\le_D$ can be lowered to $\Pi^0_1$.

Another part of the paper is devoted to the classical reducibility $\le_Q$, defined by
\[
A\le_Q B \iff (\exists \text{ total computable } f)(\forall x)\,[x\in A \iff W_{f(x)}\subseteq B].
\]
Although $\le_Q$ does not belong to the framework of reducibilities strictly below $\le_T$ on all subsets of $\omega$, it is natural to include it in a broader study of arithmetical introimmunity.
Indeed, $\le_m\Rightarrow\le_Q$, while $\le_Q$ and $\le_T$ are incomparable on $2^\omega$.
On the class of computably enumerable sets one has $\le_Q\Rightarrow\le_T$, but this is irrelevant here, since $Q$-introimmune sets must be immune.
The behavior of introimmunity for $\le_Q$ turns out to be sharply different from the reducibilities considered above:
every $\Pi^0_1$ set $A$ is $Q$-reducible to each subset $B\subseteq A$, so no infinite $\Pi^0_1$ set is $Q$-introimmune.
On the other hand, in Section \ref{sec:Q} we construct a $Q$-introimmune set in $\Delta^0_2$.
Thus the existence of $\Delta^0_2$ $Q$-introimmune sets is best possible within the arithmetical hierarchy.
Finally, in Section \ref{sec:e} we turn to enumeration reducibility $\le_e$.
Here the picture is sharply different from the arithmetical one described above.
Using the characterization of $\Pi^1_1$ sets by uniform c.e. moduli from \cite{ghtpt21}, we show that no infinite $\Pi^1_1$ set is $e$-introimmune; in particular, there are no arithmetical or hyperarithmetical $e$-introimmune sets.
On the other hand, an application of Soare's abstract theorem for countable classes of Borel measurable partial maps yields the existence of $e$-introimmune sets in the unrestricted sense \cite{so69}.
Thus enumeration reducibility admits introimmune sets, but only beyond the $\Pi^1_1$ level.

Taken together, the results of this paper yield the following picture.
For the reducibilities below $\le_T$ treated here, namely $\le_{wtt}$, $\le_{bs}$, $\le_D$, and $\le_{D^+}$, arithmetical introimmune sets do exist, and the optimal bound is obtained for $\le_{wtt}$.
The reducibility $\le_Q$ also admits arithmetical introimmune sets, with the sharp bound $\Delta^0_2$.
By contrast, $e$-introimmune sets exist only abstractly: they do exist, but no infinite $\Pi^1_1$ set is $e$-introimmune.
This considerably broadens the current existence picture for introimmune sets and highlights the descriptive diversity of the phenomenon across different reducibilities.

\paragraph*{Conventions and notation.}
Our notation and basic background from computability theory are standard. For general references, see, e.g., 
Downey--Hirschfeldt~\cite{DH2010}, Odifreddi~\cite{OdifreddiCRT, OdifreddiCRT2} and Soare~\cite{Soare1987}.

Throughout the paper, sets are identified with their characteristic functions.
We fix standard effective enumerations $\{W_e\}_{e\in\omega}$ of the computably enumerable sets, $\{\Phi_e\}_{e\in\omega}$ of oracle Turing functionals, and $\{\varphi_e\}_{e\in\omega}$ of partial computable functions.
We write $K=\emptyset'$ for the halting set, use $\langle \cdot,\cdot\rangle$ for a fixed computable pairing function, and for an infinite set $X=\{x_0<x_1<\cdots\}$ let $p_X(n)=x_n$ denote its principal function.

\section[A wtt-introimmune set in Pi01]{A $wtt$-introimmune set in $\Pi^0_1$}
\label{sec:wtt}

\subsection*{Weak truth-table reducibility}
In this section we settle the open question left in \cite{ci18} by proving that weak truth-table introimmunity already occurs at the $\Pi^0_1$ level.
The main difficulty is that, in a purely subtractive construction, the separation needed to control bounded uses cannot be fixed in advance and must instead be created dynamically during the construction.

\begin{definition}[Weak truth-table reducibility]\label{def:wtt}
Let $A,B\subseteq\omega$. We say that $A$ is \emph{weak truth-table reducible} to $B$, and write $A\le_{wtt} B$,
if there exist a Turing functional $\Phi$ and a total computable function $g$ such that for every $x$,
\[
\Phi^{B}(x)\downarrow = A(x)\;\text{and}\;
\text{every oracle query made in the computation }\Phi^{B}(x)\text{ is }\le g(x).
\]
Equivalently, $A\le_{wtt} B$ iff there exist indices $a,b$ such that $\varphi_b$ is total and, for all $x$,
$\Phi_a^{B}(x)\downarrow=A(x)$ with every oracle query in this computation $\le \varphi_b(x)$.
In this case we say that the pair $(\Phi_a,\varphi_b)$ is a $wtt$-reduction of $A$ to $B$.
\end{definition}

\begin{remark}
The reducibility $\le_{wtt}$ is also known as \emph{bounded Turing reducibility} $\le_{bT}$.
\end{remark}

We can now state and prove our main existence theorem.

\begin{theorem}\label{thm:main}
There exists a $wtt$-introimmune set in $\Pi^0_1$.
\end{theorem}

We construct a set $A$ satisfying the following requirements for every $e=\langle a,b\rangle$:
\begin{displaymath}
N_{e}: (\Phi_a,\varphi_b)\text{ does not $wtt$-reduce } A \text{ to any } X\subseteq A \text{ with } \lvert A\setminus X\rvert=\infty.
\end{displaymath}
We do not need explicit positive requirements because the finite injury nature of our subtractive construction intrinsically guarantees that $A$ is infinite.

\paragraph*{Strategy and Construction}
The set $A$ is constructed by stages $s \ge 0$. At each stage $s$ we define a cofinite set $A_s$, and the final set will be
\begin{displaymath}
A = \bigcap_{s \ge 0} A_s.
\end{displaymath}
Since $A_{s+1} \subseteq A_s$ for every $s \ge 0$ and the sequence $(A_s)_{s\ge 0}$ is uniformly computable, $A$ is a $\Pi^0_1$ set. At any stage $s$, the elements of $A_s$ are strictly ordered as
\begin{displaymath}
A_s = \{m_0^s < m_1^s < m_2^s < \dots \}.
\end{displaymath}
To satisfy the requirements, for each pair $(n,e)$ with $e \le n$ we maintain a finite set of strings $D_{n,e}^s \subseteq \{0,1\}^n$, representing the \virgolette{defeated} oracle guesses.

\vskip.2cm
{\it Algorithm}
\begin{description}
\item - Stage $0$. Set $A_0 = \omega$ (so $m_n^0 = n$ for all $n \ge 0$). Set $D_{n,e}^0 = \emptyset$ for all $e \le n$.
\item - Stage $s+1$. Let $A_s$ be the set constructed by the end of stage $s$. A pair $(n,e)$ with $e \le n \le s$ and $e = \langle a, b \rangle$ {\it is eligible to act} at stage $s+1$ if $\varphi_{b,s}(m_k^s) \downarrow$ for all $k \le n$.
Let $u=\max_{k\le n}\varphi_{b,s}(m_k^s)$ for this pair.
The pair requires attention if one of the following two conditions holds:
\begin{description}
\item - \textbf{C1}: $A_s \cap (m_n^s, u] \neq \emptyset$ (which is equivalent to $m_{n+1}^s \le u$).
\item - \textbf{C2}: Condition \textbf{C1} does not hold, and there exists a string $\alpha \in \{0,1\}^n \setminus D_{n,e}^s$ such that, letting $X_{\alpha,s}\subseteq\omega$ be the \emph{total} set defined by
\[
X_{\alpha,s}(y)=
\begin{cases}
\alpha(k) & \text{if } y=m_k^s \text{ for some } k<n,\\
0 & \text{otherwise,}
\end{cases}
\]
we have $\Phi_{a,s}^{X_{\alpha,s}}(m_k^s)\downarrow=1$ for all $k\le n$, and every oracle query in these computations is $\le u$
(hence the outcome depends only on $X_{\alpha,s}\!\upharpoonright\!(u+1)$).
\end{description}
Find the lexicographically least pair $(n,e)$ with $e \le n \le s$ that requires attention. If no such pair exists, set $A_{s+1} = A_s$, $D_{j,i}^{s+1} = D_{j,i}^s$ for all $j,i$, and go to the next stage.
Otherwise, act according to the condition satisfied by $(n,e)$:
\begin{description}
\item - \textbf{Action 1}: If $(n,e)$ requires attention via \textbf{C1}, set $A_{s+1} = A_s \setminus (m_n^s, u]$. Set $D_{n,e}^{s+1} = D_{n,e}^s$. Let $y_{\min} = m_{n+1}^s$ be the minimum element extracted.
\item - \textbf{Action 2}: If $(n,e)$ requires attention via \textbf{C2} with some string $\alpha$ (pick the lexicographically least if there are many), set $A_{s+1} = A_s \setminus \{ m_n^s \}$. Set $D_{n,e}^{s+1} = D_{n,e}^s \cup \{\alpha\}$. Let $y_{\min} = m_n^s$ be the minimum element extracted.
\end{description}

{\it Reset Rule}: For all pairs $(j,i)$ with $j \ge 1$ and $i\le j$ such that $m_{j-1}^s \ge y_{\min}$, set $D_{j,i}^{s+1} = \emptyset$. For all other pairs $(j,i)$, if it is not the pair explicitly updated by Action 2, carry over $D_{j,i}^{s+1} = D_{j,i}^s$.
\end{description}
{\it End of algorithm}

\begin{lemma}\label{lem:marker-stabilizes}
Every marker $m_n^s$ stabilizes, i.e., $\lim_{s \rightarrow \infty} m_n^s = m_n < \infty$. Consequently, $A$ is infinite.
\end{lemma}
\begin{proof}
We proceed by induction on $n$. For $n=0$ the claim is immediate: the only relevant pair is $(0,0)$ and Action 2 can occur at most once since $\{0,1\}^0$ has size 1. Hence $m^s_0$ stabilizes.
Now assume $n \ge 1$ and proceed by induction. Let $s_0$ be a stage such that the previous marker $m_{n-1}^s$ has reached its final stable value $m_{n-1}$.
This implies that for all $s \ge s_0$, no element less than or equal to $m_{n-1}$ is ever extracted from $A_s$. Therefore, for any extraction at stage $s+1 > s_0$, the minimum extracted element $y_{\min}$ is strictly greater than $m_{n-1}$. By the Reset Rule, since $m_{n-1} < y_{\min}$, the set $D_{n,e}$ is never reset to $\emptyset$ after stage $s_0$ for any $e \le n$.

We count how many times $m_n^s$ can change (i.e., increase) after $s_0$. The value $m_n^s$ increases if and only if an element $y \le m_n^s$ is extracted. Since $y \ge y_{\min} > m_{n-1}$ (and $A_s$ contains no elements strictly between $m_{n-1}$ and $m_n^s$), the minimum element extracted by any action affecting $m_n^s$ must be exactly $m_n^s$. Let us check which actions have $y_{\min} = m_n^s$:
\begin{description}
\item - Action 1 for a pair $(k,e)$ extracts elements with minimum $m_{k+1}^s$. This equals $m_n^s$ if and only if $k = n-1$. Since $m_0, \dots, m_{n-1}$ are stable, the combined use $u = \max_{j \le n-1} \varphi_{b,s}(m_j)$ can transition from divergent to convergent at most once (because $\varphi_b$ is a standard partial computable function, and once it halts on stable inputs, its value is fixed permanently). Thus Action 1 for $(n-1,e)$ clears the interval at most once for each $e \le n-1$, after which Condition \textbf{C1} never holds again for $(n-1,e)$.
\item - Action 2 for a pair $(k,e)$ extracts $m_k^s$. This equals $m_n^s$ if and only if $k = n$. When Action 2 for $(n,e)$ is executed, a string $\alpha \in \{0,1\}^n$ is permanently added to $D_{n,e}$. Since $D_{n,e}$ is never reset after $s_0$, and there are exactly $2^n$ strings of length $n$, this action can occur at most $2^n$ times for each $e \le n$.
\item - Any other action (Action 1 for $k \ge n$, Action 2 for $k > n$) has a minimum extracted element strictly greater than $m_n^s$, leaving $m_n^s$ unchanged.
\item - Action 2 for $k < n$ would extract $m_k \le m_{n-1}$, which contradicts the stability of $m_{n-1}$, so it never happens after $s_0$.
\item - Similarly, Action 1 for $k<n-1$ would extract elements with minimum $m^s_{k+1}\le m_{n-1}$, contradicting the stability of $m_{n-1}$, so it never happens after $s_0$.
\end{description}
Therefore, after $s_0$, the marker $m_n^s$ is extracted at most $n + (n+1)2^n$ times. Hence $\lim_{s \rightarrow \infty} m_n^s$ exists and is finite.
\end{proof}

\begin{lemma}\label{lem:wtt-req}
For every $e$, $N_e$ is satisfied. Thus $A$ is $wtt$-introimmune.
\end{lemma}
\begin{proof}
Let $e = \langle a, b\rangle$. If $\varphi_b$ is not a total computable function, then by definition the pair $(\Phi_a, \varphi_b)$ cannot be a valid $wtt$-reduction. In this case, the requirement $N_e$ is vacuously satisfied. Therefore, from now on we may assume that $\varphi_b$ is total.

For the sake of contradiction, assume that the requirement $N_e$ is not satisfied. This implies that there exists an $X \subseteq A$ such that $\lvert A \setminus X\rvert = \infty$ and $\Phi_a^X = A$ with the queries bounded by the total computable function $\varphi_b$.
Since $X$ omits infinitely many elements of $A$, and $A = \{m_0 < m_1 < m_2 < \dots \}$ is infinite by Lemma \ref{lem:marker-stabilizes}, there exists an integer $n \ge e$ such that $X(m_n) = 0$, where $m_n = \lim_{s \rightarrow \infty} m_n^s$.
Let $s_0$ be a sufficiently large stage such that for all $k \le n+1$ and $s \ge s_0$, $m_k^s = m_k$.
Since we have established that $\varphi_b$ is total, the computation of $\varphi_b(m_k)$ must eventually converge for all $k \le n$. Let $u = \max_{k \le n} \varphi_b(m_k)$.
Let $s_1 \ge s_0$ be a stage such that $\max_{k \le n}\varphi_{b, s_1}(m_k) \downarrow = u$.

Since all markers up to $m_{n+1}$ are stable from stage $s_0$, no pair lexicographically strictly less than $(n,e)$ can require attention at any stage $s \ge s_1$, because any such pair would act by extracting an element $\le m_{n+1}$, contradicting its stability.
Hence, if \textbf{C1} held for $(n,e)$ at some stage $s \ge s_1$, then $(n,e)$ would be the highest priority pair requiring attention at that stage and would act via Action~1, extracting $m_{n+1}$ and contradicting the stability of $m_{n+1}$.
Therefore \textbf{C1} is false for all $s \ge s_1$, so the interval $(m_n,u]$ is completely empty in $A$ and hence empty in $X \subseteq A$.

Let $\alpha \in \{0,1\}^n$ be the string defined by $\alpha(k) = X(m_k)$ for every $k < n$.
Define the (total) set $X_{\alpha}\subseteq\omega$ by $X_\alpha(m_k)=\alpha(k)$ for $k<n$ and $X_\alpha(y)=0$ for all other $y$.
(At stage $s$ we only use its restriction to $[0,u_s]$, where $u_s=\max_{k\le n}\varphi_{b,s}(m_k^s)$.)
We show that $X_{\alpha}$ matches the true oracle $X$ up to $u$:
\begin{description}
\item - For $y = m_k$ with $k < n$: $X_{\alpha}(y) = \alpha(k) = X(y)$.
\item - For $y = m_n$: $X_{\alpha}(y) = 0$ (by definition) and $X(y) = 0$ (by the choice of $n$).
\item - For all other $y \le u$: $y \notin \{m_0, \dots, m_n\}$, hence either $y<m_0$, or $m_{k-1}<y<m_k$ for some $1\le k\le n$, or $y\in (m_n,u]$.
In all cases $y\notin A$, hence $y\notin X$, so $X(y)=0=X_\alpha(y)$.
\end{description}
Therefore, $X$ exactly coincides with $X_{\alpha}$ up to the use bound $u$.
Since we are assuming $\Phi_a^X=A$ and $m_k\in A$ for all $k\le n$, the true computation $\Phi_a^X(m_k)$ must output 1.
Since the true computation $\Phi_a^X(m_k)$ only queries elements $\le \varphi_b(m_k)\le u$ and $X$ agrees with $X_{\alpha}$ up to $u$, the simulated computation
$\Phi_a^{X_{\alpha}}(m_k)$ is identical. Thus, it also outputs 1 and only queries elements $\le u$, fulfilling the condition \textbf{C2}.
Hence, the computation predicate in \textbf{C2} is satisfied via $\alpha$ for $(n,e)$ for all sufficiently large stages $s \ge s_1$.

Fix a stage $s_2 \ge s_1$ large enough that the computation predicate in \textbf{C2} is witnessed at stage $s_2$ via $\alpha$.
Then, if $\alpha \notin D_{n,e}^{s_2}$, the pair $(n,e)$ would act via Action~2 at stage $s_2+1$, contradicting the stability of $m_n$.
Hence, $\alpha \in D_{n,e}^{s_2}$.

Let $t < s_2$ be such that $\alpha \in D_{n,e}^{t+1}\setminus D_{n,e}^t$ and $t$ is maximal with this property.
Then stage $t+1$ executed Action~2 for $(n,e)$ via $\alpha$, removing
\[
x := m_n^t
\]
from $A_{t+1}$.

Since $\alpha \in D_{n,e}^{s_2}$ and $t$ was chosen maximal, $D_{n,e}$ is not reset at any stage between $t+1$ and $s_2$.
Hence no stage between $t+1$ and $s_2$ extracts an element $\le m_{n-1}^t$, so the markers $m_0^t,\dots,m_{n-1}^t$ are already their final values
$m_0,\dots,m_{n-1}$.
Therefore, $\alpha$ correctly codes $X$ on these elements.

Because $x$ is removed at stage $t+1$, we have $x\notin A$, so $A(x)=0$, and also $X(x)=0$ since $X\subseteq A$.
Moreover, since Action~2 applies at stage $t+1$, condition \textbf{C1} evaluated on $A_t$ must be false.
Let
\[
u_t := \max_{k\le n}\varphi_{b,t}(m_k^t).
\]
Then $A_t\cap (x,u_t]=\emptyset$. Furthermore, by the definition of the markers, $A_t$ contains no elements strictly between $m_{k-1}^t$ and $m_k^t$ for any $k\le n$. Since $X\subseteq A\subseteq A_t$, it follows that $X$ has no elements in $(x,u_t]$ nor strictly between any markers up to $x$.
Thus, $X$ exactly agrees with $X_\alpha$ up to $u_t$.
Finally, since \textbf{C2} holds at stage $t+1$ via $\alpha$, we have $\Phi_{a,t}^{X_\alpha}(x)\downarrow=1$.
(Notice that since $m_k^t=m_k$ for $k<n$, the evaluated set $X_{\alpha,t}$ exactly coincides with the $X_{\alpha}$ defined above, so we have $\Phi_{a,t}^{X_{\alpha}}(x)\downarrow =1$).
Since the use of $\Phi_a^X(x)$ is $\le\varphi_b(x)\le u_t$ and $X$ agrees with $X_\alpha$ below $u_t$,
it follows that $\Phi_a^X(x)=\Phi_a^{X_\alpha}(x)=1$, contradicting $A(x)=0$.
\end{proof}

\begin{proof}[Proof of Theorem \ref{thm:main}]
By Lemma~\ref{lem:marker-stabilizes}, the constructed set $A$ is infinite.
By Lemma~\ref{lem:wtt-req}, every requirement $N_e$ is satisfied.
Hence $A$ is $wtt$-introimmune.
\end{proof}

\section{Introimmunity for bounded-search and Dartmouth reducibilities in \texorpdfstring{$\Delta^0_2$}{Delta02}}
\label{sec:bs_D}

We now turn to two further reducibilities strictly below $\le_T$, namely Jockusch's bounded-search reducibility and Andersen's Dartmouth reducibility.
Since neither of them comes with a computable a priori bound on the relevant oracle use, the $\Pi^0_1$ strategy from the weak truth-table case must be replaced by a $\Delta^0_2$ construction in which the necessary spacing is computed dynamically from the halting problem.

\begin{definition}[Bounded-search and Dartmouth reducibilities]
Let $A, B \subseteq \omega$. For each index $e$, and input $n$, if $\Phi_e^B(n)\downarrow$, let $Q_e^B(n)$ denote the finite set of all numbers $u$ such that the membership or non-membership of $u$ in $B$ is used in the computation of $\Phi_e^B(n)$.
\begin{enumerate}[label=(\roman*)]
\item $A \le_{bs} B$ if there is a Turing functional $\Phi_a$ and a total computable function $\varphi_b$ such that $\Phi_a^B = A$ and
\[
\lvert Q_a^B(x)\rvert \le \varphi_b(x)
\qquad \text{for all } x.
\]
\item $A \le_D B$ if there is a Turing functional $\Phi_a$ such that $\Phi_a^B = A$ and $\Phi_a^W$ is a total function for every computably enumerable set $W$.
\item $A \le_{D^+} B$ if $A \le_D B$ and the reduction $\Phi_a$ is a positive procedure.
Of course $\le_{D^+}\Rightarrow \le_D$, and as shown by Andersen \cite{an08}, $\le_{D} \not\Rightarrow \le_{D^+}$.
Intuitively, a positive procedure uses only positive information from the oracle to certify output $1$, and only negative information to certify output $0$.
\end{enumerate}
\end{definition}
\smallskip

\begin{remark}
We may, and do, fix once and for all an effective enumeration of oracle Turing functionals that never repeat an oracle query. Indeed, from any oracle functional one can uniformly obtain an equivalent one that stores previously obtained oracle answers and reuses them whenever the same query is requested again. Hence, for functionals in this normalized enumeration, the total number of oracle queries made in the computation of $\Phi_a^B(x)$ is exactly $\lvert Q_a^B(x)\rvert$. Accordingly, in the proof below, whenever we say that a computation makes at most $\varphi_b(x)$ oracle queries, we are simply using the bound $\lvert Q_a^B(x)\rvert \le \varphi_b(x)$ for the normalized functional.
\end{remark}
Unlike $wtt$-reductions, the uses of $bs$-reductions and $D$-reductions are not bounded by an a priori computable function. This makes a purely subtractive $\Pi^0_1$ construction highly problematic. However, by using the halting set $K \equiv_T \emptyset'$ as an oracle, we can dynamically construct the topological spacing required for the diagonalization using a finite extension argument, obtaining $\{bs, D\}$-introimmune sets in $\Delta^0_2$.

\begin{theorem}\label{thm:bs_delta2}
There exists a set $A \in \Delta^0_2$ which is $bs$-introimmune. 
\end{theorem}
\smallskip

\begin{proof}
We construct $A \in \Delta^0_2$ by finite extensions, satisfying the following requirements for all $a,b,e \in \omega$:
\begin{align*}
P_{2e}:\ & \lvert A\rvert \geq e,\\
N_{2\langle a,b\rangle+1}:\ & (\Phi_a,\varphi_b)\text{ does not } bs\text{-reduce } A \text{ to any } X\subseteq A \text{ with } \lvert A\setminus X\rvert=\infty.
\end{align*}

\paragraph{Strategy and Dynamic Spacing.} 
The set $A$ will be constructed by stages $s=0,1,\ldots$. At stage $s$, we define a finite string $\alpha_s \in \{0,1\}^s$, letting $A_s = X_{\alpha_s} = \{h(n) : n < s \wedge \alpha_s(n)=1\}$, where $h$ is a strictly increasing function. The final set is $A = \bigcup_s A_s$.

Unlike the $\le_{wtt}$ case, the use of $bs$-reductions cannot be bounded a priori by a single total computable function. To overcome this, we define $h(s)$ \emph{dynamically} using the halting set $K \equiv_T \emptyset'$. 

Given a negative requirement $N_c$ where $c = 2\langle a, b \rangle+1$ and an input $x$, we use $K$ to compute a $K$-computable upper bound $U(c, x)$ on the use of any computation of $\Phi_a(x)$ making at most $\varphi_b(x)$ oracle queries.
Specifically, we ask $K$ whether $\varphi_b(x)\downarrow$. If not, set $U(c,x)=0$. If it converges, let $k = \varphi_b(x)$. For each $\sigma \in \{0,1\}^{\le k}$, we simulate $\Phi_a(x)$ by feeding the successive oracle answers prescribed by $\sigma$. Using $K$, we decide whether this simulation halts after asking at most $|\sigma|$ queries. Let $U(c,x)$ be the maximum query asked in any such halting simulation (or $0$ if there is none). Moreover, if $c$ is even we conventionally set $U(c,x)=0$ for every $x$.

We define the strictly increasing sequence $h(n)$ recursively:
\begin{displaymath}
h(0) = 0, \quad h(s+1) = \max \Big( \{h(s)\} \cup \{ U(c, h(m)) : c \le s,\ m \le s \} \Big) + 1.
\end{displaymath}

Because $U$ is $K$-computable, the sequence $h$ is $K$-computable. This definition enforces the following crucial structural bound:

\begin{claim}\label{claim:dynamic_bs}
Let $c=2\langle a,b\rangle+1$ be odd. If $s \geq \max(c,m)$ and for some oracle $F$ the computation $\Phi_a^F(h(m))$ halts making at most $\varphi_b(h(m))$ queries, then every query asked in this computation is $< h(s+1)$.
\end{claim}
\begin{proof}[Proof of Claim \ref{claim:dynamic_bs}]
Let $k=\varphi_b(h(m))$. The actual successive oracle answers seen in the computation form some $\sigma \in \{0,1\}^{\le k}$. By definition of $U(c,h(m))$, every query asked in the computation is at most $U(c,h(m))$. Since $c \le s$ and $m \le s$, the recursive definition of $h$ gives $U(c,h(m)) < h(s+1)$. Hence every query asked in the computation is $< h(s+1)$.
\end{proof}

\paragraph{Construction.}
Given any string $\alpha$, let $X_{\alpha} = \{h(n) : n < |\alpha| \wedge \alpha(n)=1\}$.
Fix a stage $s+1$, having constructed $\alpha_s$ of length $s$. Requirement $P_{2e}$ requires attention if $|X_{\alpha_s}| < e$. Requirement $N_c$ (with $c \le s$) requires attention via a string $\alpha \in \{0,1\}^s$ if:
\begin{enumerate}[label=\textbf{C\arabic*}]
    \item $\alpha(m) \le \alpha_s(m)$ for all $m < s$ (i.e., $X_{\alpha} \subseteq X_{\alpha_s}$);
    \item for every $m < s$, $\Phi_a^{X_{\alpha}}(h(m))\downarrow = \alpha_s(m)$ and the computation halts making $\le \varphi_b(h(m))$ queries;
    \item $\Phi_a^{X_{\alpha 0}}(h(s))\downarrow = 1$ and the computation halts making $\le \varphi_b(h(s))$ queries.
\end{enumerate}
A requirement is \emph{active} if it is the highest priority requirement requiring attention.

\begin{itemize}
    \item \textbf{Stage 0:} $\alpha_0 = \emptyset$.
    \item \textbf{Stage $s+1$:} Use $K$ to compute $h(s+1)$ and let $R_n$ be the least requirement with $n \le 2(s+1)$ that requires attention at stage $s+1$. Such an $n$ exists, since $P_{2(s+1)}$ requires attention: indeed, $|X_{\alpha_s}| \le s < s+1$. If $n$ is even, set $\alpha_{s+1} = \alpha_s 1$; if $n$ is odd, set $\alpha_{s+1} = \alpha_s 0$.
\end{itemize}
We set $A = \bigcup_s X_{\alpha_s}$. Since the search for active requirements and the generation of $h(s)$ only require finitely many queries to $K$, $A \le_T K$, so $A \in \Delta^0_2$.

\paragraph{Verification.}
We prove by induction on the priority that every requirement acts only finitely often and is met.

First let $R_n=P_{2e}$ be even, and assume inductively that every requirement $R_d$ with $d<n$ acts only finitely often. Choose a stage $s_0>n$ such that no requirement $R_d$ with $d<n$ acts after stage $s_0$. If $|X_{\alpha_{s_0}}|\ge e$, then $P_{2e}$ is already satisfied. Otherwise, at every later stage $t+1$ with $|X_{\alpha_t}|<e$, the requirement $P_{2e}$ requires attention. Since no smaller-index requirement acts after $s_0$, it is active whenever it requires attention, so each such action appends a $1$ and increases $|X_{\alpha_t}|$ by one. Therefore $P_{2e}$ acts only finitely often and is eventually satisfied.

Now fix an odd requirement $N_c$, say $c=2\langle a,b\rangle+1$, and assume inductively that every requirement $R_d$ with $d<c$ acts only finitely often. Choose a stage $s_0>c$ such that no requirement $R_d$ with $d<c$ acts after stage $s_0$.

For each $t \in \omega$, let $V_t^c$ be the set of all strings $\beta \in \{0,1\}^t$ such that:
\begin{enumerate}[label=(V\arabic*)]
    \item $\beta(m) \le \alpha_t(m)$ for all $m<t$;
    \item for every $m<t$, $\Phi_a^{X_{\beta}}(h(m))\downarrow = \alpha_t(m)$ and the computation halts making $\le \varphi_b(h(m))$ queries.
\end{enumerate}
Thus $N_c$ requires attention at stage $t+1$ via $\beta$ if and only if $\beta \in V_t^c$ and
\[
\Phi_a^{X_{\beta 0}}(h(t))\downarrow = 1
\]
with at most $\varphi_b(h(t))$ queries.

\begin{claim}\label{claim:prefix_validity}
If $s_0 \le s < t$ and $\beta \in V_t^c$, then $\beta{\upharpoonright}s \in V_s^c$.
\end{claim}
\begin{proof}
Condition (V1) is immediate because $\alpha_t$ extends $\alpha_s$.

Fix $m<s$. Since $\beta \in V_t^c$, the computation $\Phi_a^{X_{\beta}}(h(m))$ halts making at most $\varphi_b(h(m))$ queries. Because $s \ge s_0 > c$ and $m<s$, we have $\max(c,m) \le s-1$. 
Claim~\ref{claim:dynamic_bs}, applied with $s-1$ in place of $s$, shows that every query asked in this computation is $< h(s)$.

Now $X_{\beta{\upharpoonright}s}$ and $X_{\beta}$ agree below $h(s)$, so
\[
\Phi_a^{X_{\beta{\upharpoonright}s}}(h(m)) = \Phi_a^{X_{\beta}}(h(m)) = \alpha_t(m) = \alpha_s(m).
\]
Thus (V2) holds as well, and $\beta{\upharpoonright}s \in V_s^c$.
\end{proof}

\begin{claim}\label{claim:no_branching_valid}
Let $s \ge s_0$, let $\alpha \in V_s^c$, and suppose that $\beta \in V_t^c$ extends $\alpha$ for some $t>s$. Then
\[
\beta(s)=\alpha_{s+1}(s).
\]
\end{claim}
\begin{proof}
If $\alpha_{s+1}(s)=0$, then $\alpha_t(s)=0$ for every $t>s$, and since $\beta \in V_t^c$, condition (V1) yields
\[
\beta(s) \le \alpha_t(s)=0.
\]
Hence $\beta(s)=0$.

Suppose now that $\alpha_{s+1}(s)=1$, and assume for contradiction that $\beta(s)=0$. Since $\beta \in V_t^c$, condition (V2) at $m=s$ gives
\[
\Phi_a^{X_{\beta}}(h(s))\downarrow = \alpha_t(s)=1
\]
with at most $\varphi_b(h(s))$ queries.
Because $s \ge s_0 > c$, Claim~\ref{claim:dynamic_bs} shows that every query asked in this computation is $< h(s+1)$.

Since $\beta$ extends $\alpha$ and $\beta(s)=0$, the sets $X_{\beta}$ and $X_{\alpha 0}$ agree below $h(s+1)$. Therefore
\[
\Phi_a^{X_{\alpha 0}}(h(s))\downarrow = 1
\]
with at most $\varphi_b(h(s))$ queries. As $\alpha \in V_s^c$, this means that $\alpha$ satisfies \textbf{C1}--\textbf{C3} at stage $s+1$, so $N_c$ requires attention at stage $s+1$.

If some requirement $R_d$ with $d<c$ required attention at stage $s+1$, then by priority it would act at stage $s+1$, contradicting the choice of $s_0$. Hence no such $R_d$ requires attention, and therefore $N_c$ is active at stage $s+1$. This forces $\alpha_{s+1}(s)=0$, contradicting our assumption that $\alpha_{s+1}(s)=1$.

Therefore $\beta(s)=1$, as required.
\end{proof}

\begin{claim}\label{claim:kill_valid_branch}
If $N_c$ acts at stage $s+1$ via some $\alpha \in V_s^c$, then no $\beta \in V_t^c$ with $t>s$ extends $\alpha$.
\end{claim}
\begin{proof}
Since $N_c$ acts at stage $s+1$, we have $\alpha_{s+1}(s)=0$, and therefore $\alpha_t(s)=0$ for every $t>s$.

Suppose that $\beta \in V_t^c$ extends $\alpha$ for some $t>s$. By (V1),
\[
\beta(s) \le \alpha_t(s)=0,
\]
so $\beta(s)=0$.

Because $N_c$ acted via $\alpha$, condition \textbf{C3} held at stage $s+1$, i.e.
\[
\Phi_a^{X_{\alpha 0}}(h(s))\downarrow = 1
\]
with at most $\varphi_b(h(s))$ queries.
Again Claim~\ref{claim:dynamic_bs} shows that every query asked in this computation is $< h(s+1)$.

Since $\beta$ extends $\alpha$ and $\beta(s)=0$, the sets $X_{\beta}$ and $X_{\alpha 0}$ agree below $h(s+1)$. Hence
\[
\Phi_a^{X_{\beta}}(h(s)) = \Phi_a^{X_{\alpha 0}}(h(s)) = 1.
\]
But because $\beta \in V_t^c$, condition (V2) at $m=s$ requires
\[
\Phi_a^{X_{\beta}}(h(s)) = \alpha_t(s)=0,
\]
a contradiction. Therefore no such $\beta$ exists.
\end{proof}

Claims~\ref{claim:prefix_validity} and \ref{claim:no_branching_valid} imply that for every $t \ge s_0$, the restriction map
\[
\beta \mapsto \beta{\upharpoonright}t
\]
sends $V_{t+1}^c$ injectively into $V_t^c$. Moreover, if $N_c$ acts at stage $t+1$, then Claim~\ref{claim:kill_valid_branch} shows that some element of $V_t^c$ has no extension in $V_{t+1}^c$. Hence
\[
|V_{t+1}^c| < |V_t^c|
\]
whenever $N_c$ acts at stage $t+1$, while always $|V_{t+1}^c| \le |V_t^c|$ for $t \ge s_0$.
Since $|V_{s_0}^c| \le 2^{s_0}$, it follows that $N_c$ acts only finitely often.

It remains to show that $N_c$ is met. Suppose for contradiction that there exists $B \subseteq A$ with $\lvert A \setminus B\rvert = \infty$ such that $\Phi_a^B = A$ and, for every $x$, the computation $\Phi_a^B(x)$ makes at most $\varphi_b(x)$ queries.

Because $A \setminus B$ is infinite and $A \subseteq \{h(n): n \in \omega\}$, we may choose $s \ge s_0$ such that
\[
A(h(s))=1 \qquad \text{and} \qquad B(h(s))=0.
\]
Define $\alpha \in \{0,1\}^s$ by
\[
\alpha(m)=B(h(m)) \qquad (m<s).
\]

We claim that $\alpha \in V_s^c$. First, since $B \subseteq A$ and $\alpha_s(m)=A(h(m))$ for all $m<s$, we have
\[
\alpha(m) \le \alpha_s(m) \qquad (m<s),
\]
so (V1) holds.

Now fix $m<s$. Since $\Phi_a^B = A$, we have
\[
\Phi_a^B(h(m)) = A(h(m)) = \alpha_s(m),
\]
and this computation makes at most $\varphi_b(h(m))$ queries. 
Because $s \ge s_0 > c$ and $m<s$, Claim~\ref{claim:dynamic_bs}, applied with $s-1$ in place of $s$, shows that every query asked in this computation is $< h(s)$.

Moreover, because $B \subseteq A \subseteq \{h(n): n \in \omega\}$ and $\alpha(m)=B(h(m))$ for $m<s$, we have
\[
B \cap [0,h(s)) = X_{\alpha} \cap [0,h(s)).
\]
Therefore
\[
\Phi_a^{X_{\alpha}}(h(m)) = \Phi_a^B(h(m)) = \alpha_s(m),
\]
with at most $\varphi_b(h(m))$ queries. Thus (V2) holds, and so $\alpha \in V_s^c$.

Finally,
\[
\Phi_a^B(h(s)) = A(h(s)) = 1,
\]
again with at most $\varphi_b(h(s))$ queries.
Claim~\ref{claim:dynamic_bs} shows that every query asked in this computation is $< h(s+1)$.
Since $B(h(s))=0$ and $B \subseteq A \subseteq \{h(n): n \in \omega\}$, we have
\[
B \cap [0,h(s+1)) = X_{\alpha 0} \cap [0,h(s+1)).
\]
Hence
\[
\Phi_a^{X_{\alpha 0}}(h(s))\downarrow = 1
\]
with at most $\varphi_b(h(s))$ queries. Therefore $\alpha$ witnesses that $N_c$ requires attention at stage $s+1$.

As before, no requirement $R_d$ with $d<c$ can require attention at stage $s+1$, for otherwise it would act after stage $s_0$. Thus $N_c$ is active at stage $s+1$ and acts, forcing $\alpha_{s+1}(s)=0$. Consequently every later $\alpha_t(s)$ is also $0$, and therefore $A(h(s))=0$, contradicting the choice of $s$.

This contradiction shows that no such set $B$ exists. Hence $N_c$ is met.

By induction on the priority, every requirement is met. Therefore $A$ is $bs$-introimmune.
\end{proof}

\vspace{2em}

\begin{theorem}\label{thm:D_delta2}
There exists a set $A \in \Delta^0_2$ which is $D$-introimmune. 
\end{theorem}
(Since $\le_{D^+} \Rightarrow \le_D$, the set $A$ is automatically $D^+$-introimmune as well).

\begin{proof}
We construct $A \in \Delta^0_2$ by finite extensions, satisfying the following requirements for all $a,e \in \omega$:
\begin{align*}
P_{2e}:\ & \lvert A\rvert \geq e,\\
N_{2a+1}:\ & \Phi_a\text{ does not } D\text{-reduce } A \text{ to any } X\subseteq A \text{ with } \lvert A\setminus X\rvert=\infty.
\end{align*}

\paragraph{Strategy and Dynamic Spacing.} 
The set $A$ will be constructed by stages $s=0,1,\ldots$. At stage $s$, we define a finite string $\alpha_s \in \{0,1\}^s$, letting $A_s = X_{\alpha_s} = \{h(n) : n < s \wedge \alpha_s(n)=1\}$, where $h$ is a strictly increasing function. The final set is $A = \bigcup_{s} A_s$.

Unlike the $\le_{wtt}$ case, the use of $D$-reductions cannot be bounded a priori by a single total computable function. To overcome this, we define $h(s)$ \emph{dynamically} using the halting set $K \equiv_T \emptyset'$. 
Given a negative requirement $N_c$ where $c = 2a+1$, an input $x$, and a current bound $y$, we use $K$ to compute the maximum possible use $U(c, x, y)$ of the reduction over all valid partial oracle configurations bounded by $y$.
A valid $D$-reduction must be total on all c.e.\ sets. Since every finite set is c.e., $\Phi_a^F(x)$ must converge for every finite set $F$. There are exactly $2^{y+1}$ possible subsets $F \subseteq \{0, \dots, y\}$. Using $K$, we check if $\Phi_a^F(x)\downarrow$ for \emph{all} such subsets (assuming the oracle answers $0$ for any query $>y$). 
If it diverges for at least one, $\Phi_a$ is not a valid $D$-reduction, we set $U(c,x,y) = 0$, and requirement $N_c$ is permanently disabled.
If it converges for all of them, let $U(c,x,y)$ be the maximum use among these $2^{y+1}$ halting computations (we always set $U(c,x,y)=0$ if $c$ is even).
We define the strictly increasing sequence $h(n)$ recursively:
\begin{displaymath}
h(0) = 0, \quad h(s+1) = \max \Big( \{h(s)\} \cup \{ U(c, h(m), h(s)) : c \le s, m \le s \} \Big) + 1.
\end{displaymath}
Because $U$ is $K$-computable, the sequence $h$ is $K$-computable. This definition enforces the following crucial structural bound:

\begin{claim}\label{claim:dynamic_D}
If $N_c$ defines a valid $D$-reduction, then for every $m\in\omega$ and every stage $s\ge \max(c,m)$, the use of the computation on input $h(m)$ with any finite oracle $F \subseteq \{h(0), \dots, h(s)\}$ is strictly less than $h(s+1)$.
\end{claim}

\begin{proof}[Proof of Claim \ref{claim:dynamic_D}]
Assume that $N_c$ defines a valid $D$-reduction, and fix $m$ and $s \ge \max(c,m)$.
Let
\[
F \subseteq \{h(0),\dots,h(s)\}
\]
be finite. Since $h$ is strictly increasing, we have
\[
F \subseteq \{0,\dots,h(s)\}.
\]

Because $\Phi_a$ is a valid $D$-reduction, it is total on every c.e.\ oracle. In particular, for every finite set
\[
E \subseteq \{0,\dots,h(s)\},
\]
the computation $\Phi_a^E(h(m))$ converges. Hence, in the definition of $U(c,h(m),h(s))$, the disabling case does not occur, and
\[
U(c,h(m),h(s))
\]
is exactly the maximum use among the computations $\Phi_a^E(h(m))$ for all subsets $E \subseteq \{0,\dots,h(s)\}$.

Now our given set $F$ is one of these subsets. Therefore the use of $\Phi_a^F(h(m))$ is at most $U(c,h(m),h(s))$.

Finally, since $s \ge \max(c,m)$, we have $c \le s$ and $m \le s$, so the quantity $U(c,h(m),h(s))$ appears among the terms used to define $h(s+1)$. Thus
\[
U(c,h(m),h(s)) < h(s+1).
\]
It follows that the use of $\Phi_a^F(h(m))$ is strictly less than $h(s+1)$, as required.
\end{proof}

\paragraph{Construction.}
Given any string $\alpha$, let $X_{\alpha} = \{h(n) : n < |\alpha| \wedge \alpha(n)=1\}$.
Fix a stage $s+1$, having constructed $\alpha_s$ of length $s$. Requirement $P_{2e}$ requires attention if $|X_{\alpha_s}| < e$. Requirement $N_c$ (with $c < s$, and not permanently disabled) requires attention via a string $\alpha \in \{0,1\}^s$ if:
\begin{enumerate}[label=\textbf{C\arabic*}]
    \item for every $m < s$, $\alpha(m) \le \alpha_s(m)$;    
    \item for every $m < s$, $\Phi_a^{X_\alpha}(h(m))\downarrow = \alpha_s(m)$;
    \item $\Phi_a^{X_{\alpha 0}}(h(s))\downarrow = 1$.
\end{enumerate}
A requirement is \emph{active} if it is the highest priority requirement requiring attention.

\begin{itemize}
    \item \textbf{Stage 0:} $\alpha_0 = \emptyset$.
    \item \textbf{Stage $s+1$:} Use $K$ to compute $h(s+1)$. Let $R_n$ be the active requirement, if there is one. If no requirement requires attention, set $\alpha_{s+1}=\alpha_s1$. If $R_n=P_{2e}$, set $\alpha_{s+1}=\alpha_s1$. If $R_n=N_c$, then $N_c$ requires attention via some string $\alpha\in\{0,1\}^s$, and we act by setting $\alpha_{s+1}=\alpha_s0$.\end{itemize}
We set $A = \bigcup_s X_{\alpha_s}$. 
Since at each stage the search for active requirements and the computation of $h(s+1)$ require only finitely many queries to $K$, the construction is $K$-computable. Hence $A\le_T K$, and therefore $A\in\Delta^0_2$.

\paragraph{Verification.}
We verify by induction on the priority ordering that each requirement acts only finitely often and is satisfied.

Fix $e$. Let $s_0$ be a stage after which no requirement of higher priority than $P_{2e}$ acts. If $|A_{s_0}| \ge e$, then $P_{2e}$ is already met. Otherwise, for every stage $s \ge s_0$ with $|A_s|<e$, the requirement $P_{2e}$ is the active requirement, so we set $\alpha_{s+1}=\alpha_s1$. Each such action increases $|A_s|$ by one, and later stages never remove elements. Hence after finitely many stages we obtain $|A_s|\ge e$, and $P_{2e}$ is permanently satisfied.

Now fix $c=2a+1$. Let $s_0>c$ be a stage after which no requirement of higher priority than $N_c$ acts. If during the computation of some $U(c,h(m),h(s))$ we find a finite oracle $F$ such that $\Phi_a^F(h(m))\uparrow$, then $N_c$ is permanently disabled, and there is nothing to prove. So assume this never happens. Then for every $t\ge s_0$, every $m<t$, and every finite oracle $F \subseteq \{h(0),\dots,h(t-1)\}$, the computation $\Phi_a^F(h(m))$ has use $<h(t)$, by the definition of $U(c,h(m),h(t-1))$ and of $h(t)$.

For $t\ge s_0$, call a string $\gamma\in\{0,1\}^t$ \emph{viable at stage $t$} if it satisfies \textbf{C1} and \textbf{C2} with respect to $\alpha_t$.

If $\beta\in\{0,1\}^{t+1}$ is viable at stage $t+1$ and $\gamma=\beta\upharpoonright t$, then $\gamma$ is viable at stage $t$. Indeed, \textbf{C1} is immediate. For \textbf{C2}, fix $m<t$. By the use bound above, the computation $\Phi_a^{X_\gamma}(h(m))$ has use $<h(t)$. Since $X_\beta$ and $X_\gamma$ coincide below $h(t)$, we get
\[
\Phi_a^{X_\gamma}(h(m))=\Phi_a^{X_\beta}(h(m))=\alpha_{t+1}(m)=\alpha_t(m).
\]
Thus $\gamma$ is viable.

Now let $\gamma\in\{0,1\}^t$ be viable at stage $t$. We show that $\gamma$ has at most one viable extension to length $t+1$.

If $\alpha_{t+1}(t)=0$, then any viable extension $\beta$ of $\gamma$ must satisfy $\beta(t)\le \alpha_{t+1}(t)=0$ by \textbf{C1}, so necessarily $\beta=\gamma0$.

If $\alpha_{t+1}(t)=1$, then $\gamma0$ cannot be viable at stage $t+1$. For if it were, since $X_{\gamma0}=X_\gamma$, the string $\gamma$ would satisfy \textbf{C1} and \textbf{C2} at stage $t$, and moreover $\Phi_a^{X_{\gamma0}}(h(t))=1$. Hence $N_c$ would require attention via $\gamma$ at stage $t+1$. But after stage $s_0$ no higher priority requirement acts, so $N_c$ would then be active and we would have $\alpha_{t+1}(t)=0$, a contradiction. Therefore in this case the only possible viable extension is $\gamma1$.

Let $v_t$ be the number of viable strings of length $t$. The previous paragraphs show that every viable string of length $t+1$ extends a viable string of length $t$, and that each viable string of length $t$ has at most one viable extension. Therefore
\[
v_{t+1}\le v_t
\qquad\text{for every } t\ge s_0.
\]

Moreover, if $N_c$ acts at stage $t+1>s_0$ via a viable string $\gamma\in\{0,1\}^t$, then $\alpha_{t+1}(t)=0$. Hence $\gamma1$ fails \textbf{C1}, while $\gamma0$ fails \textbf{C2} at input $h(t)$, because \textbf{C3} gives
\[
\Phi_a^{X_{\gamma0}}(h(t))=1\neq \alpha_{t+1}(t)=0.
\]
So $\gamma$ has no viable extension at stage $t+1$, and thus $v_{t+1}<v_t$. Since $v_{s_0}\le 2^{s_0}$, the requirement $N_c$ can act at most $2^{s_0}$ times after stage $s_0$. In particular, $N_c$ acts only finitely often.

We now show that $N_c$ is met. Suppose toward a contradiction that $\Phi_a$ $D$-reduces $A$ to some set $B\subseteq A$ with $\lvert A\setminus B\rvert=\infty$. Choose $s\ge s_0$ such that
\[
A(h(s))=1
\qquad\text{and}\qquad
B(h(s))=0.
\]
Define $\alpha\in\{0,1\}^s$ by
\[
\alpha(m)=B(h(m))
\qquad (m<s).
\]
We claim that $N_c$ requires attention via $\alpha$ at stage $s+1$.

Condition \textbf{C1} holds because $B\subseteq A$, so for every $m<s$,
\[
\alpha(m)=B(h(m))\le A(h(m))=\alpha_s(m).
\]

To verify \textbf{C2}, fix $m<s$. Since $X_\alpha\subseteq \{h(0),\dots,h(s-1)\}$, Claim \ref{claim:dynamic_D} applied at stage $s-1$ gives that the computation $\Phi_a^{X_\alpha}(h(m))$ has use $<h(s)$. Also, because $B\subseteq A\subseteq \operatorname{range}(h)$, the sets $X_\alpha$ and $B$ coincide below $h(s)$. Therefore
\[
\Phi_a^{X_\alpha}(h(m))
=
\Phi_a^B(h(m))
=
A(h(m))
=
\alpha_s(m).
\]

Finally, to verify \textbf{C3}, note that $B(h(s))=0$. Hence $X_{\alpha0}$ and $B$ coincide below $h(s+1)$: below $h(s)$ this follows from the definition of $\alpha$, at $h(s)$ both answers are $0$, and any element of $B$ above $h(s)$ must be equal to $h(j)$ for some $j\ge s+1$, hence is at least $h(s+1)$. By Claim \ref{claim:dynamic_D} applied at stage $s$, the computation $\Phi_a^{X_{\alpha0}}(h(s))$ has use $<h(s+1)$. Thus
\[
\Phi_a^{X_{\alpha0}}(h(s))
=
\Phi_a^B(h(s))
=
A(h(s))
=
1.
\]
So $\alpha$ satisfies \textbf{C1}--\textbf{C3}, and therefore $N_c$ requires attention at stage $s+1$.

Since no higher priority requirement acts after $s_0$, it follows that $N_c$ acts at stage $s+1$. But then $\alpha_{s+1}(s)=0$, i.e.\ $A(h(s))=0$, contradicting the choice of $s$. This contradiction shows that no such set $B$ exists. Hence $N_c$ is met.

By induction, every requirement is met. Therefore $A$ is $D$-introimmune.

\end{proof}

\section{Arithmetical Q-introimmune sets}\label{sec:Q}

This section is devoted to the classical reducibility $\le_Q$, which falls outside the framework of reducibilities strictly below $\le_T$ on all subsets of $\omega$.
Here the situation can be determined exactly at the level of arithmetical existence: infinite $\Pi^0_1$ sets are never $Q$-introimmune, whereas $\Delta^0_2$ $Q$-introimmune sets do exist.

\begin{definition}[$Q$-reducibility]
Let $A,B\subseteq\omega$. We say that $A\le_Q B$ if there exists a total computable function $f$ such that
\[
x\in A \iff W_{f(x)}\subseteq B
\qquad\text{for all } x\in\omega.
\]
\end{definition}

\begin{theorem}
Let $A \subseteq \omega$ be a $\Pi^0_1$ set. Then $A \le_Q B$ for every $B \subseteq A$. In particular, no infinite $\Pi^0_1$ set is $Q$-introimmune.
\end{theorem}

\begin{proof}
Since $A$ is $\Pi^0_1$, its complement $\overline{A}$ is c.e. Fix an index $e$ such that
\[
W_e=\overline{A}.
\]
By the $s$-$m$-$n$ theorem, there exists a total computable function $f$ such that, for each $x \in \omega$, the c.e.\ set $W_{f(x)}$ is generated by the program which waits for $x$ to appear in the enumeration of $W_e$ and, if this happens, enumerates $x$. Hence
\[
W_{f(x)}=
\begin{cases}
\varnothing & \text{if } x \in A,\\
\{x\} & \text{if } x \notin A.
\end{cases}
\]

Now let $B \subseteq A$. If $x \in A$, then $W_{f(x)}=\varnothing \subseteq B$. If $x \notin A$, then $W_{f(x)}=\{x\}$, and since $B \subseteq A$, we have $x \notin B$, so $W_{f(x)} \nsubseteq B$. Therefore,
\[
x \in A \iff W_{f(x)} \subseteq B.
\]
Since $f$ is total computable, this proves that $A \le_Q B$.

For the final claim, assume that $A$ is infinite and choose $B \subseteq A$ so that both $B$ and $A \setminus B$ are infinite. Then $A \le_Q B$ by the first part, and hence $A$ is not $Q$-introimmune.
\end{proof}

\begin{theorem}\label{thm:Q_delta2}
There exists a set $A \in \Delta^0_2$ which is $Q$-introimmune. 
\end{theorem}
\begin{proof}
We construct $A \in \Delta^0_2$ by finite extensions, using the halting set $K \equiv_T \emptyset'$ as an oracle. Let $\{V_{e,x}\}_{e,x \in \omega}$ be a standard effective enumeration of all uniformly computably enumerable families of sets. It suffices to satisfy the following requirements for all $e \in \omega$:
\begin{align*}
P_e:\ & \lvert A\rvert \geq e,\\
R_e:\ & \text{The sequence } V_{e,x} \text{ does not } Q\text{-reduce } A \text{ to any } B\subseteq A \text{ with } \lvert A\setminus B\rvert=\infty.
\end{align*}

To prevent lower-priority requirements from jumping over elements and bypassing conditions, the construction proceeds \emph{bit-by-bit}. We use a finite injury priority argument computable in $\emptyset'$. At stage $s$, we define a finite string $\sigma_s \in \{0,1\}^{<\omega}$ of length exactly $s$, letting $A_s = \{z < s : \sigma_s(z)=1\}$ be the finite set of elements already enumerated into $A$. 

Each requirement $R_e$ can be in one of four states: \emph{Initial}, \emph{Working}, \emph{Waiting}, or \emph{Satisfied}. 
When in the \emph{Waiting} state, $R_e$ stores two parameters: a finite base set $D_e$ and a length bound $L_e$.
When in the \emph{Working} state, $R_e$ stores a target finite string $\tau_e \supset \sigma_s$.

\paragraph{Construction.}
\begin{itemize}
    \item \textbf{Stage 0:} $\sigma_0 = \emptyset$. All requirements are in the \emph{Initial} state.
    \item \textbf{Stage $s+1$:} Let $x' = s$. Find the least $e \le s$ such that $R_e$ is not \emph{Satisfied} and requires attention. $R_e$ requires attention if one of the following conditions holds:
    \begin{enumerate}[label=\textbf{C\arabic*}]
        \item $R_e$ is \emph{Initial} and $\exists x \ge x' \exists y \notin A_s \cup \{x\} \, (y \in V_{e,x})$. Since $V_{e,x}$ is c.e.\ and $A_s$ is a known finite set, this is a $\Sigma^0_1$ condition, which can be evaluated using $\emptyset'$.
        \item $R_e$ is \emph{Initial} and \textbf{C1} is false.
        \item $R_e$ is \emph{Waiting} (with parameters $D_e, L_e$). By the failure of \textbf{C1} when $R_e$ entered the \emph{Waiting} state, we know that $V_{e,x'} \subseteq D_e \cup \{x'\}$. Using $\emptyset'$, we can exactly compute the finite set $V_{e,x'}$. Condition \textbf{C3} holds if $x' \notin V_{e,x'}$ and the set $S = V_{e,x'} \cap D_e$ is \virgolette{recorded}. We say $S$ is recorded if there exists some $x \in [L_e, x')$ such that $\sigma_s(x) = 1$, $x \notin V_{e,x}$, and $V_{e,x} \cap D_e = S$.
        \item $R_e$ is \emph{Working} (with target string $\tau_e$). This naturally holds as long as $x' < \lvert\tau_e\rvert$.
    \end{enumerate}
    
    If no such $e$ exists, set $\sigma_{s+1} = \sigma_s \frown 1$ (to implicitly satisfy the positive requirements $P_k$), and proceed to the next stage. Otherwise, let $e$ be the least such index, and act according to the condition satisfied:
    \begin{itemize}
        \item \textbf{Action for C1:} Use $\emptyset'$ to find the lexicographically least pair $(x, y)$ witnessing \textbf{C1}. Let $\tau_e$ be the unique extension of $\sigma_s$ of length $\max(x,y)+1$ such that $\tau_e(x) = 1$, $\tau_e(y) = 0$ (if $y \ge x'$), and $\tau_e(z) = 0$ for all other $z \in [x', \max(x,y)+1) \setminus \{x,y\}$. Set $\sigma_{s+1} = \sigma_s \frown \tau_e(x')$. If $\lvert\sigma_{s+1}\rvert = \lvert\tau_e\rvert$, set $R_e$ to \emph{Satisfied}; otherwise, set $R_e$ to \emph{Working} with target $\tau_e$. Initialize all $R_i$ for $i > e$ (i.e., reset them to \emph{Initial}).
        \item \textbf{Action for C2:} Set $R_e$ to \emph{Waiting} with $D_e = A_s$ and $L_e = x'$. Set $\sigma_{s+1} = \sigma_s \frown 1$, and initialize all $R_i$ for $i > e$.
        \item \textbf{Action for C3:} Set $\sigma_{s+1} = \sigma_s \frown 0$ (so $x' \notin A$). Set $R_e$ to \emph{Satisfied}, and initialize all $R_i$ for $i > e$.
        \item \textbf{Action for C4:} Set $\sigma_{s+1} = \sigma_s \frown \tau_e(x')$. If $\lvert\sigma_{s+1}\rvert = \lvert\tau_e\rvert$, set $R_e$ to \emph{Satisfied}. Initialize all $R_i$ for $i > e$.
    \end{itemize}
\end{itemize}
We set $A = \bigcup_s A_s$. Since $\lvert\sigma_s\rvert = s$ exactly at each stage, the construction is well-defined. All decisions and searches at stage $s+1$ are computable from $\emptyset'$, hence $A \in \Delta^0_2$.

\paragraph{Verification.}
By standard finite injury, each $R_e$ is initialized only finitely often. After its last initialization, $R_e$ starts in \emph{Initial}. It can act via \textbf{C1}, after which it acts finitely many times via \textbf{C4} (advancing bit-by-bit) until it becomes \emph{Satisfied}. Alternatively, it acts via \textbf{C2} (entering \emph{Waiting}). Once in the \emph{Waiting} state, its parameters $D_e, L_e$ are permanently fixed, and it may later act at most once via \textbf{C3}, becoming \emph{Satisfied}.
Thus, every requirement eventually stabilizes. It remains to verify that $A$ is infinite. If infinitely many stages use the default action, then infinitely many $1$'s are appended, so $A$ is infinite. Otherwise, infinitely many stages are devoted to actions of requirements. Since each fixed requirement acts only finitely many times, infinitely many distinct requirements act. For each such requirement, consider its first action after its last initialization. That action is either \textbf{C2}, which immediately writes a $1$, or \textbf{C1}. In the latter case, the chosen target string $\tau_e$ contains some coordinate $x$ with $\tau_e(x)=1$, and since no further initialization of $R_e$ occurs afterwards, that $1$ is eventually written into the construction, either immediately at the \textbf{C1} stage or at a later \textbf{C4} stage when $\tau_e$ is completed. Hence infinitely many $1$'s are written into the characteristic function of $A$, so $A$ is infinite. Therefore every $P_e$ is satisfied.
To see that $R_e$ is met, suppose for a contradiction that the family $(V_{e,x})_{x\in\omega}$ $Q$-reduces $A$ to a coinfinite subset $B \subseteq A$. We analyze the final state of $R_e$:
\begin{itemize}
    \item If $R_e$ is \emph{Satisfied} via \textbf{C1}/\textbf{C4}: The target string $\tau_e$ was fully written. We forced $x \in A$ and $y \notin A$, with $y \in V_{e,x}$. Since $B \subseteq A$, $y \notin B$, so $V_{e,x} \not\subseteq B$. However, $x \in A$, so the $Q$-reduction would require $V_{e,x} \subseteq B$. Contradiction.
    \item If $R_e$ is \emph{Satisfied} via \textbf{C3}: We forced $x' \notin A$, and there is some recorded element $x \in [L_e, x')$ with $x \in A$ such that $V_{e,x} \cap D_e = V_{e,x'} \cap D_e = S$. Since both $x, x' \ge L_e$ and \textbf{C1} globally failed for $R_e$ when it entered \emph{Waiting}, we have $V_{e,x} \subseteq D_e \cup \{x\}$ and $V_{e,x'} \subseteq D_e \cup \{x'\}$. Since $x \notin V_{e,x}$ and $x' \notin V_{e,x'}$, this implies exactly $V_{e,x} = S$ and $V_{e,x'} = S$. Since $x \in A$, the reduction implies $V_{e,x} \subseteq B$, so $S \subseteq B$. But $x' \notin A$, which implies $V_{e,x'} \not\subseteq B$, so $S \not\subseteq B$. Contradiction.
    \item If $R_e$ remains \emph{Waiting} forever: Then \textbf{C3} never triggers. By the failure of \textbf{C1} when $R_e$ entered the \emph{Waiting} state, we have
    \[
    V_{e,x} \subseteq D_e \cup \{x\}
    \qquad\text{for every } x \ge L_e.
    \]
    For every $x>L_e$, the stage deciding whether $x$ enters $A$ evaluates \textbf{C3} for $R_e$; the only exceptional value is $x=L_e$, which was put into $A$ when $R_e$ entered the \emph{Waiting} state via \textbf{C2}. This causes no problem. Indeed, if $x\ge L_e$, $x\in A$, and $x\notin V_{e,x}$, let
    \[
    S_x = V_{e,x}\cap D_e.
    \]
    Since $x\notin V_{e,x}$ and $V_{e,x}\subseteq D_e\cup\{x\}$, we actually have $V_{e,x}=S_x\subseteq D_e$.
    For each fixed subset $S\subseteq D_e$, at most the least $x\ge L_e$ with
    \[
    x\in A,\qquad x\notin V_{e,x},\qquad V_{e,x}\cap D_e=S
    \]
    can belong to $A$: once such an $x$ exists, the set $S$ is recorded. If any later $x'>x$ existed with $x'\not\in V_{e,x'}$ and $V_{e,x'}\cap D_e=S$, it would satisfy condition \textbf{C3}, causing $R_e$ to act and become \emph{Satisfied}. This contradicts the assumption that $R_e$ remains \emph{Waiting} for ever. Thus, no such $x'$ can exist (and in particular, no such $x'$ can belong to $A$).
    Hence there are at most $2^{|D_e|}$ many $x \ge L_e$ such that $x \in A$ and $x \notin V_{e,x}$. Since there are only finitely many $x < L_e$, it follows that only finitely many $x \in A$ satisfy $x \notin V_{e,x}$. Therefore, for almost all $x \in A$, we have $x \in V_{e,x}$.
    By the $Q$-reduction, $x\in A$ implies $V_{e,x}\subseteq B$, and hence $x\in B$ for almost all $x\in A$. This means that $A\setminus B$ is finite, contradicting that $B$ is coinfinite in $A$.\end{itemize}
In all cases, the reduction fails. Hence $A$ is $Q$-introimmune.
\end{proof}

\section[Enumeration reducibility: nonexistence at the Pi11 level and abstract existence]{Enumeration reducibility: nonexistence at the \texorpdfstring{$\Pi^1_1$}{Pi11} level and abstract existence}\label{sec:e}

We now turn to enumeration reducibility \cite{fr59}.
In contrast with the previous reducibilities, the situation for $\le_e$ is not arithmetical: we first show that no infinite $\Pi^1_1$ set is $e$-introimmune, and hence that there are no arithmetical or hyperarithmetical examples.
We then show, however, that $e$-introimmune sets do exist in the unrestricted sense.

Recall that for sets $A,B\subseteq\omega$ one writes $A\le_e B$ if there exists a computably enumerable set $W\subseteq \omega\times [\omega]^{<\omega}$ such that
\[
x\in A \iff (\exists F\subseteq B\text{ finite})\; \langle x,F\rangle\in W.
\]
For an infinite set $X=\{x_0<x_1<\cdots\}$, let $p_X(n)=x_n$ denote its principal function.

The next theorem gives a negative result for $\le_e$ analogous in spirit to the nonexistence theorems of Jockusch and Simpson for Turing introimmunity \cite{jo73,si78}. 

The key inputs are \cite[Proposition~2.3]{ghtpt21}, which characterizes $\Pi^1_1$ sets by uniform c.e. moduli, and \cite[Proposition~2.1]{ghtpt21}, which identifies uniform relative c.e. procedures with uniform enumeration operators.

\begin{theorem}\label{thm:no-pi11-e-introimmune}
Let $A\subseteq\omega$ be infinite and $\Pi^1_1$. Then there exists a set $B\subseteq A$ such that $\lvert A\setminus B\rvert =\infty$ and, moreover, there is a single enumeration operator $\Psi$ satisfying
\[
\Psi(C)=A
\qquad\text{for every infinite } C\subseteq B.
\]
In particular, $A\le_e B$.
\end{theorem}

\begin{proof}
By \cite[Proposition~2.3]{ghtpt21}, there exist a function $f\in\omega^\omega$ and a relative c.e.\ operator $\Theta$ such that
\[
\Theta(g)=A
\qquad\text{for every total function } g\geq f.
\]

Let $p_A$ be the principal function of $A$. Since $p_A$ is strictly increasing and unbounded, we may choose inductively integers
\[
n_0<n_1<n_2<\cdots
\]
such that
\[
p_A(n_s)\geq f(s)
\qquad\text{and}\qquad
n_{s+1}>n_s+1
\]
for every $s\in\omega$. Set
\[
B=\{p_A(n_s):s\in\omega\}.
\]
Then $B\subseteq A$ is infinite and its principal function satisfies
\[
p_B(s)=p_A(n_s)\geq f(s)
\qquad\text{for all } s.
\]
Also, $A\setminus B$ is infinite, because for each $s$ we have $n_s+1<n_{s+1}$, hence
\[
p_A(n_s+1)\in A\setminus B.
\]

We claim that $A$ is uniformly c.e.\ in every infinite subset of $B$. Let $C\subseteq B$ be infinite. Then
\[
p_C\geq p_B\geq f,
\]
so
\[
\Theta(p_C)=A.
\]

To make the uniformity explicit, fix a c.e.\ set of axioms defining $\Theta$, that is, a c.e.\ set
\[
E\subseteq \omega^{<\omega}\times\omega
\]
such that
\[
x\in\Theta(g)\iff (\exists \sigma\prec g)\;(\sigma,x)\in E.
\]
For each strictly increasing finite sequence
\[
\sigma=(u_0<\cdots<u_{k-1})\in\omega^{<\omega}
\]
(with the convention that $\tau_\emptyset=\emptyset$), let $\tau_\sigma\in 2^{<\omega}$ be the binary string of length $u_{k-1}+1$ whose $1$'s occur exactly at the positions $u_0,\dots,u_{k-1}$.
Define
\[
\widehat{E}=\{(\tau_\sigma,x): \sigma \text{ is strictly increasing and } (\sigma,x)\in E\},
\]
and let $\widehat{\Theta}$ be the relative c.e.\ operator on set oracles determined by $\widehat{E}$.
Then, for every infinite set $C\subseteq\omega$,
\[
\widehat{\Theta}(C)=\Theta(p_C),
\]
viewing $C$ as its characteristic function, because for every strictly increasing $\sigma$, one has  $\sigma\prec p_C$ if and only if $\tau_\sigma\prec C$.
Therefore, for every infinite $C\subseteq B$,
\[
\widehat{\Theta}(C)=\Theta(p_C)=A.
\]
So $A$ is uniformly c.e.\ in every infinite subset of $B$.

By \cite[Proposition~2.1]{ghtpt21}, there exists a single enumeration operator $\Psi$ such that
\[
\Psi(C)=A
\qquad\text{for every infinite } C\subseteq B.
\]
In particular, taking $C=B$, we obtain $A\le_e B$.
\end{proof}

\begin{corollary}\label{cor:no-arith-hyp-e-introimmune}
No infinite $\Pi^1_1$ set is $e$-introimmune. Consequently, there are no arithmetical $e$-introimmune sets and no hyperarithmetical $e$-introimmune sets.
\end{corollary}

\begin{proof}
Let $A$ be infinite and $\Pi^1_1$. By Theorem~\ref{thm:no-pi11-e-introimmune}, there is a set $B\subseteq A$ such that $\lvert A\setminus B\rvert=\infty$ and $A\le_e B$. Hence $A$ is not $e$-introimmune. The final statement follows because every arithmetical or hyperarithmetical set is $\Delta^1_1$, hence in particular $\Pi^1_1$.
\end{proof}

\begin{remark}
Theorem~\ref{thm:no-pi11-e-introimmune} actually yields more than the failure of $e$-introimmunity: after passing to a suitable coinfinite subset $B\subseteq A$, the set $A$ is uniformly enumeration reducible from \emph{every} infinite subset of $B$.
\end{remark}

Thus $e$-introimmunity disappears already below the $\Pi^1_1$ level.
It is therefore natural to ask whether the notion is vacuous or whether non-effective examples still exist.
The next theorem shows that $e$-introimmune sets do exist, by an application of Soare's general existence theorem \cite{so69} to the countable family of enumeration operators.

\begin{theorem}\label{thm:e-existence}
There exists an $e$-introimmune set.
\end{theorem}

\begin{proof}
Let $\{\Psi_e\}_{e\in\omega}$ be a standard effective enumeration of all enumeration operators, and view each $\Psi_e$ as a total map from $2^\omega$ to $2^\omega$.

We first observe that each $\Psi_e$ is Borel measurable. Fix $e$ and $x$, and let
\[
U_{e,x}=\{X\in 2^\omega : x\in \Psi_e(X)\}.
\]
If $E_e\subseteq \omega\times[\omega]^{<\omega}$ is the c.e.\ set of axioms defining $\Psi_e$, then
\[
U_{e,x}=\bigcup_{\langle x,F\rangle\in E_e}\{X\in 2^\omega : F\subseteq X\},
\]
hence $U_{e,x}$ is open. Therefore, if $[\sigma]\subseteq 2^\omega$ is a basic clopen cylinder, then
\[
\Psi_e^{-1}([\sigma])
=
\bigcap_{\substack{n<|\sigma|\\ \sigma(n)=1}} U_{e,n}
\;\cap\;
\bigcap_{\substack{n<|\sigma|\\ \sigma(n)=0}} (2^\omega\setminus U_{e,n}),
\]
which is Borel. Thus $\Psi_e$ is Borel measurable.

Applying Soare's theorem to the countable family $\{\Psi_e\}_{e\in\omega}$ and to $A=\omega$, we obtain an infinite set $B\subseteq\omega$ such that
\[
(\forall e)(\forall S,T)\bigl[S,T\subseteq B \text{ infinite } \& \Psi_e(S)=T \rightarrow T\setminus S \text{ is finite}\bigr].
\]
We claim that $B$ is $e$-introimmune.

Suppose toward a contradiction that $B\le_e C$ for some $C\subseteq B$ with $\lvert B\setminus C\rvert=\infty$. Let $e$ be such that $\Psi_e(C)=B$.

If $C$ is infinite, then taking $S=C$ and $T=B$ contradicts the displayed property, since $B\setminus C$ is infinite.

So $C$ must be finite. But then $B$ is computably enumerable, because
\[
x\in B \iff (\exists F\subseteq C)\ \langle x,F\rangle\in E_e,
\]
and there are only finitely many subsets of $C$. Choose any infinite coinfinite set $R\subseteq B$.
Since $B$ is c.e., the set of axioms $\{\langle x,\emptyset\rangle:x\in B\}$ is c.e., which defines an enumeration operator which ignores the oracle and simply enumerates $B$; hence $B\le_e R$. Now $R$ is infinite, so taking $S=R$ and $T=B$ again contradicts the displayed property, because $B\setminus R$ is infinite.

This contradiction shows that no such $C$ exists. Therefore $B$ is $e$-introimmune.
\end{proof}

\begin{remark}
Theorem~\ref{thm:e-existence} yields more than the existence of a single $e$-introimmune set.
The set $B$ obtained in the proof has the stronger property that for all infinite $S,T\subseteq B$, if $T\le_e S$, then $T\setminus S$ is finite.
In particular, the same argument shows that every infinite subset of $B$ is itself $e$-introimmune.
Combined with Corollary~\ref{cor:no-arith-hyp-e-introimmune}, this shows that enumeration reducibility admits $e$-introimmune sets only in a genuinely non-effective form: they do exist, but no infinite $\Pi^1_1$ set can be one.
\end{remark}

\section{Conclusions and Future Work}
\label{sec:conclusions}

The aim of this paper has been to clarify how low introimmune sets can occur for several reducibilities arising in computability theory.
For reducibilities strictly below Turing reducibility, we obtained three arithmetical existence results: a $wtt$-introimmune set in $\Pi^0_1$, a $bs$-introimmune set in $\Delta^0_2$, and a $D$-introimmune set in $\Delta^0_2$; consequently, a $\Delta^0_2$ $D^+$-introimmune set exists as well.
The first of these results answers the open question left in \cite{ci18} and shows that $\Pi^0_1$ is the optimal arithmetical level for weak truth-table introimmunity, since $\Sigma^0_1$ sets cannot be immune.

We also analyzed the classical reducibility $\le_Q$, showing that no infinite $\Pi^0_1$ set is $Q$-introimmune while a $\Delta^0_2$ $Q$-introimmune set does exist.
Thus the existence of $\Delta^0_2$ $Q$-introimmune sets is best possible within the arithmetical hierarchy.
Finally, for enumeration reducibility $\le_e$, we obtained a qualitatively different picture: no infinite $\Pi^1_1$ set is $e$-introimmune, so there are no arithmetical, hyperarithmetical, or even $\Pi^1_1$ examples; nevertheless, by applying Soare's abstract theorem to the countable family of enumeration operators, we proved that $e$-introimmune sets do exist.
Thus, in the enumeration case, introimmunity survives only in a genuinely non-effective form.

Taken together, these results considerably enlarge the current existence picture for introimmune sets.
They show that the descriptive complexity of introimmune sets is highly sensitive to the form of oracle access allowed by the reducibility.
For $\le_{wtt}$, the rigid computable use can still be handled inside a purely subtractive $\Pi^0_1$ construction by means of the dynamic spacing technique developed in Section \ref{sec:wtt}.
For $\le_{bs}$ and $\le_D$, by contrast, no computable a priori use bound is available, and the required spacing must be computed with the aid of the halting problem; this is why our present constructions for these reducibilities live in $\Delta^0_2$ rather than in $\Pi^0_1$.
For $\le_Q$, a bit-by-bit finite-extension construction yields a sharp $\Delta^0_2$ existence theorem together with a $\Pi^0_1$ obstruction.
For $\le_e$, the phenomenon is no longer arithmetical at all: the nonexistence theorem is obtained from the descriptive-set-theoretic characterization of $\Pi^1_1$ sets by uniform c.e.\ moduli, whereas abstract existence follows from Soare's theorem for countable families of Borel maps.

A natural next problem is to understand the degree-theoretic distribution of $\Pi^0_1$ $wtt$-introimmune sets.
It is known \cite{ci21} that every non-zero c.e.\ degree computes a $\Pi^0_1$ $m$-introimmune set.
It is therefore natural to ask whether an analogous ubiquity phenomenon holds for $wtt$-introimmunity.
Our present construction suggests serious obstacles to a straightforward permitting argument: the dynamic spacing method relies on preserving carefully prepared gaps, whereas external permissions may arrive at stages that allow an opponent to destroy backup configurations infinitely often.
This leads us to conjecture that, unlike the many-one case, there are non-zero c.e.\ Turing degrees that compute no $\Pi^0_1$ $wtt$-introimmune set.

Another natural direction is to sharpen the results obtained for the other reducibilities studied here.
For $\le_{bs}$ and $\le_D$, the main open question is whether the present $\Delta^0_2$ existence results can be lowered to $\Pi^0_1$.
For $\le_Q$, since the minimal arithmetical existence level has now been identified, it would be interesting to investigate what additional structural restrictions can be imposed on $\Delta^0_2$ $Q$-introimmune sets.
For $\le_e$, where no infinite $\Pi^1_1$ example can exist, a natural problem is to locate the least possible descriptive complexity of $e$-introimmune sets above the $\Pi^1_1$ level.

These problems suggest that introimmunity remains closely connected with the fine structure of reducibilities, and that even small changes in the definition of reducibility may lead to markedly different arithmetical and descriptive existence phenomena.

\section*{Acknowledgments}

This work is the result of an extended human–AI collaboration. Several of the main structural ideas and technical arguments emerged from exploratory interaction with AI-based reasoning systems: Gemini Deep Think (Google DeepMind) and ChatGPT Pro (OpenAI). The author has fully reworked and verified all arguments and bears sole responsibility for the correctness of the results.

\end{document}